\documentclass[12pt, reqno, twoside, letterpaper]{amsart}

\usepackage{paperstyle}

\title[On certain zeta functions associated with Beatty sequences]
      {On certain zeta functions\\ associated with Beatty sequences}
      
\author[W.\ D.\ Banks]{William D.\ Banks}

\address{Department of Mathematics, 
         University of Missouri, 
         Columbia MO, USA.}

\email{bankswd@missouri.edu}
        
\date{\today}

\begin{document}

\begin{abstract}
Let $\alpha>1$ be an irrational number of finite type $\tau$.
In this paper, we introduce and study a zeta function
$Z_\alpha^\sharp(r,q;s)$ that is closely related
to the Lipschitz-Lerch zeta function and is naturally
associated with the Beatty sequence
$\cB(\alpha)\defeq(\fl{\alpha m})_{m\in\NN}$.
If $r$ is an element of the lattice $\ZZ+\ZZ\alpha^{-1}$,
then $Z_\alpha^\sharp(r,q;s)$ continues analytically to the
half-plane $\{\sigma>-1/\tau\}$ with its only singularity
being a simple pole at $s=1$.  If $r\not\in\ZZ+\ZZ\alpha^{-1}$,
then $Z_\alpha^\sharp(r,q;s)$ extends analytically to the
half-plane $\{\sigma>1-1/(2\tau^2)\}$ and has no singularity
in that region.

\end{abstract}

\maketitle

\begin{center}
\emph{To the memory of Tom Apostol}
\end{center}


\section{Introduction and statement of results}

The \emph{Lipschitz-Lerch zeta function} is defined in the half-plane
$\{\sigma\defeq\Re s>1\}$ by an absolutely convergent series
$$
\zeta(z,q;s)\defeq\sum_{n=0}^\infty\frac{\e(zn)}{(n+q)^s},
$$
where $\e(t)\defeq e^{2\pi it}$ for all $t\in\RR$, and by analytic
continuation it extends to meromorphic function on the whole $s$-plane.
The function $\zeta(z,q;s)$ was introduced by Lipschitz~\cite{Lip1} for real
$z$ and $q>0$; see also Lipschitz~\cite{Lip2}.  It
also bears the name of Lerch~\cite{Lerch}, who showed that
for $\Im z>0$ and $q\in(0,1)$ the functional equation
$$
\zeta(z,q;1-s)=(2\pi)^{-s}\Gamma(s)\(
\e(\tfrac14s-zq)\zeta(-q,z;s)+
\e(-\tfrac14s+zq)\zeta(q,1-z;s)\)
$$
holds; this is called \emph{Lerch's transformation formula}.  For an
interesting account of the analytic properties of the Lipschitz-Lerch zeta function
and related functions, we refer the reader to the work
of Lagarius and Li~\cite{LagLi1,LagLi2,LagLi3,LagLi4}; see
also Apostol~\cite{Apostol}.

If $z\in\ZZ$, then $\zeta(z,q;s)=\zeta(0,q;s)$ is the
\emph{Hurwitz zeta function}; in this case, $\zeta(z,q;s)$ has a
simple pole at $s=1$ but no other singularities
in the $s$-plane.  On the other hand, if $z\in\RR\setminus\ZZ$ or $\Im z>0$,
then $\zeta(z,q;s)$ is an entire function of $s$.

For a given real number $\alpha>0$, the \emph{homogeneous Beatty sequence}
associated with $\alpha$ is the sequence of natural numbers defined by
$$
\cB(\alpha)\defeq(\fl{\alpha m})_{m\in\NN},
$$
where $\fl{\cdot}$ denotes the floor function: $\fl{t}$ is the greatest
integer $\le t$ for any $t\in\RR$.
Beatty sequences appear in a wide variety of unrelated
mathematical settings, and their
arithmetic properties have been extensively
explored in the literature;
see, for example,
\cite{Abe,BaSh0,BaSh1,BaSh2,Beg,FraHolz,Guo,Kom1,Kom2,LuZh,OB,Tijd} and the
references therein.

In this paper, we introduce and study
a variant of the Lipschitz-Lerch zeta function that is naturally
associated with the Beatty sequence $\cB(\alpha)$.
Specifically, let us denote
\begin{equation}
\label{eq:Zarqs-defn0}
Z_\alpha(r,q;s)\defeq\sum_{n\in\cB(\alpha)}\frac{\e(rn)}{(n+q)^s},
\end{equation}
where $r\in\RR$ and $q\in(0,1)$.
For technical reasons, the work in this paper is focused on
properties of the function 
$$
Z_\alpha^\sharp(r,q;s)\defeq e^{\pi ir}Z_\alpha(r,q;s)
+e^{-\pi ir}Z_\alpha(-r,1-q;s).
$$
Further, we assume that $\alpha>1$ is irrational and of \emph{finite type}
(see \S\ref{sec:disc-type}).  Note that for rational $\alpha$, the Beatty sequence
$\cB(\alpha)$ is a finite union of arithmetic progressions, and therefore
the analytic properties of $Z_\alpha(r,q;s)$ and $Z_\alpha^\sharp(r,q;s)$
can be gleaned from well known properties of the
Lipschitz-Lerch zeta function.

The series \eqref{eq:Zarqs-defn0} converges absolutely
the half-plane $\{\sigma>1\}$, uniformly on compact regions,
hence $Z_\alpha(r,q;s)$ is analytic
there; this implies that $Z_\alpha^\sharp(r,q;s)$ is analytic in
$\{\sigma>1\}$ as well.

Since $\cB(\alpha)$ is a set of density $\alpha^{-1}$ in the set of
natural numbers, it is reasonable to expect that $Z_\alpha(r,q;s)$
is closely related to the function $\alpha^{-1}\zeta(r,q;s)$.
This belief is strengthened by the fact that if $\alpha^{-1}+\beta^{-1}=1$,
then the set of natural numbers can be split as the disjoint union
of $\cB(\alpha)$ and $\cB(\beta)$, so we have
$$
Z_\alpha(r,q;s)+Z_\beta(r,q;s)+q^{-s}
=\alpha^{-1}\zeta(r,q;s)+\beta^{-1}\zeta(r,q;s).
$$
Naturally, one might also expect that $Z_\alpha^\sharp(r,q;s)$ is closely related
to the function $\alpha^{-1}\zeta^\sharp(r,q;s)$, where
$$
\zeta^\sharp(r,q;s)\defeq e^{\pi ir}\zeta(r,q;s)
+e^{-\pi ir}\zeta(-r,1-q;s).
$$
As it turns out, such expectations are erroneous.  Our first theorem
establishes that $Z_\alpha^\sharp(r,q;s)$ has a simple pole at $s=1$
whenever $r$ is an element of the lattice $\ZZ+\ZZ\alpha^{-1}$; this
lattice is a \emph{dense} subset of $\RR$.  By contrast, the function
$\alpha^{-1}\zeta^\sharp(r,q;s)$, being a linear combination of 
Lipschitz-Lerch zeta functions, can only have a pole when $r$ is an
integer.

\begin{theorem}
\label{thm:main}
Let $\alpha>1$ be an irrational number of finite type $\tau$.
Suppose that $r=k\alpha^{-1}+\ell$ for some integers $k$ and $\ell$,
and let $q\in(0,1)$.  Then the function
\begin{equation}
\label{eq:toTom}
Z_\alpha^\sharp(r,q;s)-\alpha^{-1}\zeta^\sharp(r,q;s)
\end{equation}
continues analytically to the half-plane $\{\sigma>-1/\tau\}$
with a simple pole at $s=1$ and no other singularities.
The residue at $s=1$ is $2(-1)^\ell\alpha^{-1}$ when $k=0$,
and it is $2(-1)^\ell(\sin(\pi k\alpha^{-1}))/(\pi k)$ for $k\ne 0$.
\end{theorem}

In particular, taking $r\defeq 0$ and $q\defeq\tfrac12$ above, we have
$$
Z_\alpha^\sharp(0,\tfrac12;s)\defeq 2 Z_\alpha(0,\tfrac12;s)
=\sum_{n\in\cB(\alpha)}\frac{2}{(n+\tfrac12)^s}
$$
and
$$
\zeta^\sharp(0,\tfrac12;s)\defeq 2\zeta(0,\tfrac12;s)
=\sum_{n=0}^\infty\frac{2}{(n+\tfrac12)^s}=(2^{s+1}-2)\zeta(s),
$$
where $\zeta(s)$ is the \emph{Riemann zeta function} studied by
Riemann \cite{Riemann} in 1859.  Since $\zeta(s)$
has a simple pole at $s=1$ with residue one, from Theorem~\ref{thm:main}
we deduce the following corollary.

\begin{corollary}
\label{cor:one}
Let $\alpha>1$ be irrational of finite type $\tau$.
The Dirichlet series
$$
Z_\alpha(0,\tfrac12;s)\defeq\sum_{n\in\cB(\alpha)}\frac{1}{(n+\tfrac12)^s}
$$
converges absolutely in $\{\sigma>1\}$ and extends
analytically to $\{\sigma>-1/\tau\}$, where it has
a simple pole at $s=1$ with residue $2\alpha^{-1}$
and no other singularities.
\end{corollary}

When $r\defeq 0$ and $q\defeq\tfrac12$
our proof of Theorem~\ref{thm:main} shows that the function
$$
F_\alpha^\sharp(0,\tfrac12;s)
=\pi^{-s/2}\Gamma(s/2)\big(2 Z_\alpha(0,\tfrac12;s)
-\alpha^{-1}(2^{s+1}-2)\zeta(s)+2^s\big)
$$
is analytic at $s=0$.  Since $\Gamma(s/2)$ has a pole at $s=0$,
one sees that $Z_\alpha(0,\tfrac12;s)$ takes the same value
at $s=0$ regardless of the choice of $\alpha$.

\begin{corollary}
\label{cor:one}
For every 
irrational $\alpha>1$ of finite type,
$Z_\alpha(0,\tfrac12;0)=-\tfrac12$.
\end{corollary}

Our second theorem is complementary to Theorem~\ref{thm:main};
it establishes that $Z_\alpha^\sharp(r,q;s)$ does not have a pole at $s=1$
whenever the real number $r$ is not contained in the lattice $\ZZ+\ZZ\alpha^{-1}$.

\begin{theorem}
\label{thm:main2}
Let $\alpha>1$ be an irrational number of finite type $\tau$.
Let $q\in(0,1)$, and suppose that
$r$ is a real number \underline{not} of the form $k\alpha^{-1}+\ell$
with $k,\ell\in\ZZ$. Then the function \eqref{eq:toTom}
continues analytically to the half-plane
$\big\{\sigma>1-1/(2\tau^2)\big\}$
with no singularities in that region.
\end{theorem}

\section{Preliminaries}

\subsection{General notation}

Throughout the paper, we fix an irrational number $\alpha>1$
of finite type $\tau=\tau(\alpha)$ (see \S\ref{sec:disc-type}),
and we set $\gamma\defeq\alpha^{-1}$.
Note that $\gamma$ has the same type $\tau$.

As stated earlier, we write $\e(t)\defeq e^{2\pi it}$ for all $t\in\RR$.
We use $\fl{t}$ and $\{t\}$ to denote the greatest integer
not exceeding $t$ and the fractional part of $t$, respectively.
The notation $\distint{t}$ is used to represent the distance
from the real number~$t$ to the nearest integer; in other words,
\begin{equation}
\label{eq:strangethings}
\distint{t}\defeq\min_{n\in\ZZ}|t-n|=\min\big(\{t\},1-\{t\}\big)\qquad(t\in\RR).
\end{equation}
For every real number $t$, we denote by $\nearint{t}$ the integer that lies
\emph{closest} to $t$ if $t\not\in\tfrac12+\ZZ$, and we put
$\nearint{t}\defeq\fl{t}$ if $t\in\tfrac12+\ZZ$.  Then
\begin{equation}
\label{eq:distnear}
t=\begin{cases}
\nearint{t}+\distint{t}&\quad\hbox{if $\{t\}\le\tfrac12$},\\
\nearint{t}-\distint{t}&\quad\hbox{if $\{t\}>\tfrac12$}.
\end{cases}
\end{equation}

In what follows, any implied constants in the symbols $O$, $\ll$
and $\gg$ \emph{may depend on the parameters} $\alpha,r,q,\eps$
but are independent of other variables unless indicated otherwise.
For given functions $F$ and $G$,
the notations $F\ll G$, $G\gg F$ and
$F=O(G)$ are all equivalent to the statement that the inequality
$|F|\le c|G|$ holds with some constant $c>0$.

\subsection{Discrepancy and type}
\label{sec:disc-type}

The \emph{discrepancy} $D(M)$ of a sequence of (not
necessarily distinct) real numbers $(a_m)_{m=1}^M$ contained
in $[0,1)$ is given by
$$
D(M)\defeq\sup_{\cI\subseteq[0,1)}\bigg|\frac{V(\cI,M)}{M}-|\cI|\,\bigg|,
$$
where the supremum is taken over all intervals $\cI$ in
$[0,1)$, $V(\cI,M)$ is the number of positive integers
$m\le M$ such that $a_m\in\cI$, and $|\cI|$ is the length of $\cI$.

The \emph{type} $\tau=\tau(\gamma)$ of a given irrational number $\gamma$
is defined by
$$
\tau\defeq\sup\big\{t\in\RR:\liminf\limits_{n\to\infty}
~n^t\distint{\gamma n}=0\big\}.
$$
Using Dirichlet's approximation theorem, one sees that
$\tau\ge 1$ for every irrational number~$\gamma$. The
theorems of Khinchin \cite{Khin} and of Roth \cite{Roth1,Roth2}
assert that $\tau=1$ for almost all real numbers (in
the sense of the Lebesgue measure) and all irrational
algebraic numbers $\gamma$, respectively; see also \cite{Bug,Schm}.

Given an irrational number $\gamma$, the sequence
of fractional parts $(\{n\gamma\})_{n=1}^\infty$
is known to be uniformly distributed in $[0,1)$ (see 
\cite[Example~2.1, Chapter~1]{KuNi}).
In the case that
$\gamma$ is of finite type, the following more precise statement holds
(see \cite[Theorem~3.2, Chapter~2]{KuNi}).

\begin{lemma}
\label{lem:discr_with_type}  Let $\gamma$ be a fixed irrational
number of finite type $\tau$.  Then, for every $\delta\in\RR$ the
discrepancy $D_{\gamma,\delta}(M)$ of the sequence
$(\{\gamma m+\delta\})_{m=1}^M$ satisfies the bound
$$
D_{\gamma,\delta}(M)\ll M^{-1/(\tau+\eps)},
$$
where the implied constant depends only on $\gamma$ and $\eps$.
\end{lemma}

\subsection{Functional equations of theta functions}

\begin{lemma}
For any real numbers $v,w$  let
\begin{equation}
\label{eq:thetadefn}
\Theta_{v,w}(u)\defeq \e(\tfrac12vw)\sum_{n\in\ZZ}e^{-\pi(n+v)^2u}\e(wn)
\qquad(u>0).
\end{equation}
Then
\begin{equation}
\label{eq:fun3}
\Theta_{v,w}(u)=u^{-1/2}\Theta_{w,-v}(u^{-1}).
\end{equation}
\end{lemma}

\begin{proof}
Let $u>0$ be fixed, and put
$$
f(x)\defeq e^{-\pi(x+v)^2u}\e(wx)\qquad(x\in\RR).
$$
The Fourier transform of $f$ is given by
$$
\hat f(x)\defeq\int_{-\infty}^\infty f(y)\e(xy) dy
=\int_{-\infty}^\infty e^{g(x,y)} dy,
$$
where
$$
g(x,y)\defeq -\pi(y+v)^2u+2\pi i(x+w)y.
$$
Since
$$
g(x,y)=-\pi(y+v-i(x+w)u^{-1})^2u
-2\pi iv(x+w)-\pi(x+w)^2u^{-1},
$$
it follows that
\begin{equation}
\label{eq:coffee}
\hat f(x)=e^{-\pi(x+w)^2u^{-1}}\e(-v(x+w))
\int_{-\infty}^\infty e^{-\pi(y+v-iw u^{-1}-ixu^{-1})^2u} dy.
\end{equation}
Using Cauchy's theorem to shift the line of integration vertically,
we see that the integral in \eqref{eq:coffee} is equal to
$$
\int_{-\infty}^\infty e^{-\pi(y+v)^2u} dy
=\int_{-\infty}^\infty e^{-\pi y^2u} dy
=u^{-1/2}\int_{-\infty}^\infty e^{-\pi y^2} dy
=u^{-1/2}.
$$
Consequently,
$$
\hat f(x)=u^{-1/2}e^{-\pi(x+w)^2u^{-1}}\e(-v(x+w))\qquad(x\in\RR).
$$
Applying the Poisson Summation Formula 
$$
\sum_{n\in\ZZ}f(n)=\sum_{n\in\ZZ}\hat f(n),
$$
we immediately deduce the functional equation \eqref{eq:fun3}.
\end{proof}

\subsection{The pulse wave}
Let $\ind{\alpha}$ denote the indicator function of $\cB(\alpha)$; that is,
\begin{equation}
\label{eq:indalphn-initial}
\ind{\alpha}(n)\defeq\begin{cases}
1&\quad\hbox{if $n=\fl{\alpha m}$ for some $m\in\NN$},\\
0&\quad\hbox{otherwise}.    
\end{cases}
\end{equation}
In this notation we have
$$
Z_\alpha(r,q;s)=\sum_{n=1}^\infty
\frac{\ind{\alpha}(n)\e(rn)}{(n+q)^s}.
$$

Let $\cX_\gamma$ be the periodic function defined by
\begin{equation}
\label{eq:Xdefn}
\cX_\gamma(t)\defeq\begin{cases}
1&\quad\hbox{if $\{t\}\in(0,\gamma)$},\\
\frac12&\quad\hbox{if $t\in\ZZ$ or $t\in\gamma+\ZZ$},\\
0&\quad\hbox{otherwise}.
\end{cases}
\end{equation}
Since $\cX_\gamma$ is periodic of bounded variation,
and $\cX_\gamma(t)=\tfrac12(\cX_\gamma(t^+)+\cX_\gamma(t^-))$
for all $t\in\RR$, its Fourier series converges everywhere, and we have
$$
\cX_\gamma(t)=\lim_{K\to\infty}
\sum_{|k|\le K}\widetilde \cX_\gamma(k)\e(kt),
$$
where the Fourier coefficients are given by
$$
\widetilde \cX_\gamma(0)\defeq\gamma\mand
\widetilde \cX_\gamma(k)\defeq\frac{1-\e(-k\gamma)}{2\pi ik}
\qquad(k\ne 0).
$$

For any irrational $\alpha>1$, it is easy to see that
a natural number $n$ lies in the Beatty
sequence $\cB(\alpha)$ if and only if $\{-n\gamma\}\in(0,\gamma)$.
Using this characterization, the indicator function $\ind{\alpha}$ given by
\eqref{eq:indalphn-initial} satisfies
\begin{equation}
\label{eq:indalphn-2}
\ind{\alpha}(n)=\cX_\gamma(-n\gamma)=\lim_{K\to\infty}
\sum_{|k|\le K}\widetilde \cX_\gamma(k)\e(-kn\gamma).
\end{equation}
From now on, we regard $\ind{\alpha}$ as a function on all of $\ZZ$
by defining the value $\ind{\alpha}(n)$ at an arbitrary integer $n$ via the
relation \eqref{eq:indalphn-2}.

Using our hypothesis that $\alpha$ is of finite type, the
relation \eqref{eq:indalphn-2} can be made more explicit; namely, 
for any positive real number $K$, we have the estimate
\begin{equation}
\label{eq:indalphn-precise}
\ind{\alpha}(n)=
\sum_{|k|\le K}\widetilde \cX_\gamma(k)\e(-kn\gamma)
+O\big(\max\{1,|n|^{\tau+\eps}\}K^{-1}\big)\qquad(n\in\ZZ)
\end{equation}
for any given $\eps>0$. Indeed, for each nonzero integer $n$ let
$$
S_K(n;u)\defeq\sum_{k\le u}(1-\e(-k\gamma))\e(-kn\gamma)\qquad(u>0).
$$
Using standard estimates for exponential sums
(see, e.g., Korobov~\cite{Kor}) we have
$$
S_K(n;u)\ll\distint{n\gamma}^{-1}+\distint{n\gamma+\gamma}^{-1}
$$
Since $\gamma$ is of type $\tau$, this implies that the bound
$$
S_K(n;u)\ll |n|^{\tau+\eps}
$$
holds, and therefore
$$
\sum_{k>K}\widetilde \cX_\gamma(k)\e(-kn\gamma)
=\int_K^\infty\frac{dS_K(n;u)}{2\pi iu}
\ll\frac{|n|^{\tau+\eps}}{K}.
$$
Bounding $\sum_{k>K}\widetilde \cX_\gamma(-k)\e(kn\gamma)$ in a similar manner,
we deduce \eqref{eq:indalphn-precise} in the case that $n\ne 0$.
When $n=0$, we have by \eqref{eq:indalphn-2}:
$$
\ind{\alpha}(0)=\gamma+\lim_{K\to\infty}
\sum_{0<|k|\le K}\frac{\e(k\gamma)}{2\pi ik}.
$$
Writing
$$
S_K(0;u)\defeq\sum_{k\le u}\e(k\gamma)\qquad(u>0),
$$
we have $S_K(0;u)\ll\distint{\gamma}^{-1}=\alpha\ll 1$,
and therefore
$$
\sum_{k>K}\widetilde \cX_\gamma(k)
=\int_K^\infty\frac{dS_K(0;u)}{2\pi iu}
\ll K^{-1}.
$$
This yields \eqref{eq:indalphn-precise} in the case that $n=0$.

\section{The proofs}

For all $u>0$ we denote
\begin{align*}
\Psi^+(r,q;u)&\defeq \sum_{n=0}^\infty
e^{-\pi(n+q)^2u}\e(rn),\\
\Psi_\alpha^+(r,q;u)&\defeq \sum_{n=0}^\infty
e^{-\pi(n+q)^2u}\ind{\alpha}(n)\e(rn).
\end{align*}
and we also put
\begin{align*}
\Psi(r,q;u)&\defeq \sum_{n\in\ZZ}
e^{-\pi(n+q)^2u}\e(rn),\\
\Psi_\alpha(r,q;u)&\defeq \sum_{n\in\ZZ}
e^{-\pi(n+q)^2u}\ind{\alpha}(n)\e(rn).
\end{align*}
It is easy to see that
\begin{equation}
\label{eq:PsiNZbounds}
\max\big\{
\big|\Psi(r,q;u)\big|,
\big|\Psi_\alpha(r,q;u)\big|\big\}\ll e^{-\pi\distintup{q}u}
\end{equation}
holds for $u>0$, and the relations
\begin{align}
\label{eq:PsiNZreln}
\Psi(r,q;u)&=\Psi^+(r,q;u)+\e(-r)\Psi^+(-r,1-q;u),\\
\label{eq:PsialphaNZreln}
\Psi_\alpha(r,q;u)&=\Psi_\alpha^+(r,q;u)+\e(-r)\Psi_\alpha^+(-r,1-q;u),
\end{align}
are immediate. Indeed, \eqref{eq:PsiNZreln} follows from
the fact that the polynomial $(n+\tfrac12)^2$ is invariant under
the map $n\mapsto -n-1$.  To prove \eqref{eq:PsialphaNZreln},
we note that \eqref{eq:indalphn-2} implies
$$
\ind{\alpha}(n)=\gamma+\lim_{K\to\infty}
\sum_{0<|k|\le K}\frac{\sin(\pi k\gamma)}{\pi k}
\cos(2\pi k\gamma(n+\tfrac12)),
$$
hence $\ind{\alpha}(n)=\ind{\alpha}(-n-1)$ for all $n\in\ZZ$.

Next, recall that
$$
\pi^{-s/2}\Gamma(s/2)\mu^{-s}
=\int_0^\infty e^{-\pi \mu^2u} u^{s/2-1}du
\qquad(\sigma>0)
$$
for every positive real number $\mu$.  Taking into account that
$\ind{\alpha}(0)=\tfrac12$ in view of
\eqref{eq:Xdefn} and \eqref{eq:indalphn-2},  
the function
$$
F_\alpha^+(r,q;s)\defeq
\pi^{-s/2}\Gamma(s/2)\big(Z_\alpha(r,q;s)
-\gamma\zeta(r,q;s)+\tfrac12q^{-s}\big)
$$
satisfies the relation
$$
F_\alpha^+(r,q;s)=\int_0^\infty\big(\Psi_\alpha^+(r,q;u)-\gamma\Psi^+(r,q;u)\big)
u^{s/2-1}du\qquad(\sigma>1).
$$
Therefore, using \eqref{eq:PsiNZreln} and \eqref{eq:PsialphaNZreln}
it follows that
\begin{align*}
F_\alpha^\sharp(r,q;s)
&\defeq e^{\pi ir}F_\alpha^+(r,q;s)+e^{-\pi ir}F_\alpha^+(-r,1-q;s)\\
&=\pi^{-s/2}\Gamma(s/2)\big(Z_\alpha^\sharp(r,q;s)
-\gamma\zeta^\sharp(r,q;s)+\tfrac12e^{\pi ir}q^{-s}
+\tfrac12e^{-\pi ir}(1-q)^{-s}\big)
\end{align*}
satisfies the relation
$$
F_\alpha^\sharp(r,q;s)
=e^{\pi ir}\int_0^\infty\big(\Psi_\alpha(r,q;u)-\gamma\Psi(r,q;u)\big)
u^{s/2-1}du\qquad(\sigma>1).
$$
To prove Theorems~\ref{thm:main} and \ref{thm:main2},
it suffices to show that $F_\alpha^\sharp(r,q;s)$
continues analytically in an appropriate manner according
to whether or not $r$ lies in the lattice $\ZZ+\ZZ\alpha^{-1}$.

To simplify the notation, we put
$$
F(s)\defeq F_\alpha^\sharp(r,q;s)\mand
\Phi(u)\defeq\Psi_\alpha(r,q;u)-\gamma\Psi(r,q;u),
$$
and write
$$
F(s)=e^{\pi ir}\big(F_0(s)+F_\infty(s)\big),
$$
where
$$
F_0(s)\defeq\int_0^1\Phi(u)u^{s/2-1}du
\mand
F_\infty(s)\defeq\int_1^\infty\Phi(u)u^{s/2-1}du.
$$
In view of \eqref{eq:PsiNZbounds} it is clear
that the integral $F_\infty(s)$
converges absolutely and uniformly on compact regions of $\CC$,
hence $F_\infty(s)$ continues to an \emph{entire}
function of $s\in\CC$.
Thus, the analytic continuation of $F(s)$ reduces to that of $F_0(s)$.

Let $K$ be a real-valued function such that $K(u)\ge 2$ for all $u>0$.
For the moment, let $u>0$ be fixed.
Using \eqref{eq:indalphn-precise} along with the bound
$$
\sum_{n\in\ZZ} e^{-\pi(n+q)^2u}\max\{1,|n|^{\tau+\eps}\}\ll 1,
$$
we have
\begin{align*}
\Psi_\alpha(r,q;u)
&=\sum_{n\in\ZZ} e^{-\pi(n+q)^2u}\e(rn)
\sum_{|k|\le K(u)}\widetilde \cX_\gamma(k)\e(-kn\gamma)+O(K(u)^{-1}).
\end{align*}
Reversing the order of summation and recalling \eqref{eq:thetadefn},
we see that
$$
\Psi_\alpha(r,q;u)
=\sum_{|k|\le K(u)}\widetilde \cX_\gamma(k)
\e(\tfrac12q(k\gamma-r))\Theta_{q,r-k\gamma}(u)
+O(K(u)^{-1}).
$$
Since $\widetilde\cX_\gamma(0)=\gamma$, and 
$\Psi(r,q;u)=\e(-\tfrac12qr)\Theta_{q,r}(u)$ by \eqref{eq:thetadefn},
we derive the estimate
$$
\Phi(u)=\Phi_K(u)+O(K(u)^{-1}),
$$
where
\begin{equation}
\label{eq:redsolocup}
\Phi_K(u)\defeq\sum_{0<|k|\le K(u)}\widetilde \cX_\gamma(k)
\e(\tfrac12q(k\gamma-r))\Theta_{q,r-k\gamma}(u).
\end{equation}
In particular,
\begin{equation}
\label{eq:F0s-expr}
F_0(s)=\int_1^\infty\Phi_K(u^{-1})u^{-s/2-1}du
+O\bigg(\int_0^1K(u)^{-1}u^{\sigma/2-1}du\bigg).
\end{equation}
By the functional equation \eqref{eq:fun3}
and the definition \eqref{eq:thetadefn},
$$
\Theta_{q,r-k\gamma}(u)
=u^{-1/2}\Theta_{r-k\gamma,-q}(u^{-1})
=u^{-1/2}\e(\tfrac12q(k\gamma-r))
\sum_{n\in\ZZ}e^{-\pi(n+r-k\gamma)^2u^{-1}}\e(-qn).
$$
Combining this expression with \eqref{eq:redsolocup} we have
\begin{equation}
\label{eq:five}
\Phi_K(u^{-1})=u^{1/2}\sum_{0<|k|\le K(u^{-1})}\widetilde \cX_\gamma(k)
\e(q(k\gamma-r))\sum_{n\in\ZZ}e^{-\pi(n+r-k\gamma)^2u}\e(-qn).
\end{equation}

Now, we make the specific choice
$$
K(u)=K_L(u)\defeq \max\{2,u^{-\theta} L\}\qquad\text{with}\quad
\theta\defeq \frac{1-\eps}{2(\tau+\eps)},
$$
where $L$ is a large positive real number,
and $\eps>0$ is fixed (and small).  In particular, for all large $u$ we have
$$
\Phi_{K_L}(u^{-1})=u^{1/2}\sum_{0<|k|\le u^\theta L}\widetilde \cX_\gamma(k)
\e(q(k\gamma-r))\sum_{n\in\ZZ}e^{-\pi(n+r-k\gamma)^2u}\e(-qn).
$$
Note that \eqref{eq:F0s-expr} takes the form
\begin{equation}
\label{eq:F0s-expr2}
F_0(s)=\int_1^\infty\Phi_{K_L}(u^{-1})u^{-s/2-1}du
+O\big(L^{-1}(\sigma+2\theta)^{-1}\big)
\end{equation}
provided that $\sigma>-2\theta$.

To proceed further, for each integer $k$ we write
$$
r-k\gamma=\nearint{r-k\gamma}+\nu_k\distint{r-k\gamma}
$$
with $\nu_k\in\{\pm 1\}$ as in \eqref{eq:distnear}.  Making
the change of variables $n\mapsto n-\nearint{r-k\gamma}$ in the
inner summation of \eqref{eq:five}, it follows that
$$
\Phi_{K_L}(u^{-1})=u^{1/2}\sum_{0<|k|\le u^\theta L}\widetilde \cX_\gamma(k)
\e(-q\nu_k \distint{r-k\gamma})\sum_{n\in\ZZ}
e^{-\pi(n+\nu_k\distintup{r-k\gamma})^2u}\e(-qn).
$$
We introduce the notation
\begin{equation}
\label{eq:deltau-defn}
\delta_u\defeq\min\big\{\distintup{r-k\gamma}:|k|\le u^\theta L\big\},
\end{equation}
and let $\kappa_u$ denote an integer for which
$$
|\kappa_u|\le u^\theta L\mand
\distintup{r-\kappa_u\gamma}=\delta_u.
$$
Then we have
$$
\Phi_{K_L}(u^{-1})=u^{1/2}\big(\Upsilon_L^{(1)}(u)+\Upsilon_L^{(2)}(u)
+\Upsilon_L^{(3)}(u)\big),
$$
where
\begin{align*}
\Upsilon_L^{(1)}(u)
&\defeq\sum_{0<|k|\le u^\theta L}\widetilde \cX_\gamma(k)
\e(-q\nu_k \distint{r-k\gamma})\sum_{\substack{n\in\ZZ\\n\ne 0}}
e^{-\pi(n+\nu_k\distintup{r-k\gamma})^2u}\e(-qn),\\
\Upsilon_L^{(2)}(u)
&\defeq\sum_{\substack{0<|k|\le u^\theta L\\k\ne \kappa_u}}
\widetilde \cX_\gamma(k)\e(-q\nu_k \distint{r-k\gamma})
e^{-\pi\distintup{r-k\gamma}^2u},\\
\Upsilon_L^{(3)}(u)
&\defeq\widetilde \cX_\gamma(\kappa_u)\e(-q\nu_{\kappa_u}\delta_u)
e^{-\pi\delta_u^2u}.
\end{align*}
By \eqref{eq:F0s-expr2} it follows that the estimate
$$
F_0(s)=G_L^{(1)}(s)+G_L^{(2)}(s)+G_L^{(3)}(s)
+O\big(L^{-1}(\sigma+2\theta)^{-1}\big)
$$
holds provided that $\sigma>-2\theta$, where
$$
G_L^{(j)}(s)\defeq\int_1^\infty\Upsilon_L^{(j)}(u)u^{-s/2-1/2}du
\qquad(j=1,2,3).
$$
Thus, the analytic continuation of $F_0(s)$ rests on the analytic
properties of the integrals $G_L^{(j)}(s)$.

Since $\widetilde \cX_\gamma(k)\ll|k|^{-1}$
and $\distintup{r-k\gamma}\le\tfrac 12$ for every
$k\ne 0$, the bound
$$
\Upsilon_L^{(1)}(u)\ll u^{1/2}e^{-\pi u/4}\log (\max\{2,u^\theta L\})
$$
is obvious.  For fixed $L$, this implies that the integral
$G_L^{(1)}(s)$ converges absolutely for all $s\in\CC$, uniformly on
compact regions, and hence $G_L^{(1)}(s)$ is an entire function of $s$.

To determine the analytic behavior of $G_L^{(2)}(s)$, we consider
two distinct cases according to the size of $\delta_u$.

\textsc{Case 1:} $\delta_u\ge u^{-1/2}\log u$.  By the definition
of $\delta_u$ (see \eqref{eq:deltau-defn}) it follows that
$\distint{r-k\gamma}\ge u^{-1/2}\log u$
for all $k$ such that $|k|\le u^\theta L$. Since
$\widetilde \cX_\gamma(k)\ll|k|^{-1}$ for every $k\ne 0$, we have
\begin{equation}
\label{eq:case1est}
\big|\Upsilon_L^{(2)}(u)\big|
\ll\sum_{0<|k|\le u^\theta L}|k|^{-1}
e^{-\pi\distintup{r-k\gamma}^2u}
\ll\log(u^\theta L)u^{-\pi\log u}.
\end{equation}
Note that the bound
\begin{equation}
\label{eq:case1est(3)}
\big|\Upsilon_L^{(3)}(u)\big|\ll u^{-\pi\log u}
\end{equation}
also holds in this case.

\textsc{Case 2:} $\delta_u\le u^{-1/2}\log u$.  Let $k$ be
such that $0<|k|\le u^\theta L$ and $k\ne\kappa_u$. Then
\begin{align*}
r-k\gamma&=\nearint{r-k\gamma}+\nu_k\distint{r-k\gamma},\\
r-\kappa_u\gamma&=\nearint{r-\kappa_u\gamma}+\nu_{\kappa_u}\delta_u,\\
\kappa_u\gamma-k\gamma&=\nearint{\kappa_u\gamma-k\gamma}
+\nu\distint{\kappa_u\gamma-k\gamma},
\end{align*}
where $\nu\in\{\pm 1\}$ as in \eqref{eq:distnear}, and therefore
\begin{equation}
\label{eq:smaug}
\distint{r-k\gamma}\equiv \nu_k\nu_{\kappa_u}\delta_u
+\nu_k\nu\distint{\kappa_u\gamma-k\gamma}\pmod 1
\end{equation}
Noting that $|\kappa_u-k|\le 2u^\theta L$, and using the fact that $\gamma$
is of type $\tau$, we also have
\begin{equation}
\label{eq:smaug2}
\distint{\kappa_u\gamma-k\gamma}\gg|k-\kappa_u|^{-\tau-\eps}
\gg(u^\theta L)^{-\tau-\eps}=L^{-\tau-\eps}u^{-1/2+\eps/2}.
\end{equation}
As $\delta_u=o(u^{-1/2+\eps/2})$ as $u\to\infty$, from
\eqref{eq:smaug} and \eqref{eq:smaug2} we derive the lower bound
$$
\distint{r-k\gamma}\gg L^{-\tau-\eps}u^{-1/2+\eps/2}
\qquad(0<|k|\le u^\theta L,~k\ne\kappa_u).
$$
Arguing as in Case~1, this implies that
\begin{equation}
\label{eq:case2est}
\big|\Upsilon_L^{(2)}(u)\big|
\ll\log(u^\theta L)\exp(-\pi L^{-2\tau-2\eps} u^\eps).
\end{equation}

For fixed $L$, the bounds \eqref{eq:case1est} and
\eqref{eq:case2est} together imply that the integral
$G_L^{(2)}(s)$ converges absolutely for all $s\in\CC$, uniformly on
compact regions, and therefore $G_L^{(2)}(s)$ is an entire function of $s$.

Turning now 
to the analytic behavior of $G_L^{(3)}(s)$, we consider two distinct cases
according to whether or not $\delta_u$ vanishes on the interval $(1,\infty)$.

First, suppose that $\delta_u=0$ for some $u>1$.  In view of the definition
\eqref{eq:deltau-defn},
this condition is equivalent to the statement that $r=k\gamma+\ell$
for some (uniquely determined)
integers $k$ and~$\ell$.  In this case, one has $\delta_u=0$ and $\kappa_u=k$ for
all sufficiently large $u$.  In particular, for some sufficiently large
real number $U_L$, one sees that
$\Upsilon_L^{(3)}(u)=\widetilde \cX_\gamma(k)$ once $u\ge U_L$.
Consequently, if we denote
$$
G_L^{(4)}(s)\defeq G_L^{(3)}(s)-\frac{2\widetilde \cX_\gamma(k)}{s-1}
=\int_1^{U_L}\big(\Upsilon_L^{(3)}(u)-\widetilde \cX_\gamma(k)\big)u^{-s/2-1/2}du,
$$
then clearly $G_L^{(4)}(s)$ converges absolutely for all $s\in\CC$,
uniformly on compact regions, and so $G_L^{(4)}(s)$ is an entire function of $s$.
Putting everything together, we have therefore shown that
$$
F_\alpha^\natural(r,q;s)
\defeq F_\alpha^\sharp(r,q;s)-\frac{2e^{\pi ir}\widetilde \cX_\gamma(k)}{s-1}
=H_L(s)+O\big(L^{-1}(\sigma+2\theta)^{-1}\big),
$$
where
$$
H_L(s)\defeq e^{\pi ir}\big(F_\infty(s)+G_L^{(1)}(s)
+G_L^{(2)}(s)+G_L^{(4)}(s)\big)
$$
is an entire function of $s\in\CC$ for any fixed $L$.
The sequence $(H_L(s))_{L\ge 1}$ converges uniformly to
$F_\alpha^\natural(r,q;s)$ on every compact subset
of the half-plane $\{\sigma>-2\theta\}$, hence $F_\alpha^\natural(r,q;s)$
is analytic in the same region.  Since $\theta\to 1/(2\tau)$ as $\eps\to 0^+$, Theorem~\ref{thm:main} follows.

Next, we suppose that $\delta_u\ne 0$ for all $u>1$.
Observe that the map $u\to\delta_u$ is a positive nonincreasing 
step function which tends to zero as $u\to\infty$.  Let
$u_1<u_2<\cdots$ be the ordered sequence of real numbers $u_j>1$
that have one or both of the following properties:
\begin{itemize}
\item[$(i)$] $\delta_u>\delta_{u_j}$ for all $u\in(1,u_j)$,
\item[$(ii)$] $\delta_{u_j}=u_j^{-1/2}\log u_j$.
\end{itemize}
Put $u_0\defeq 1$.  Note that the sequence $(u_j)_{j\ge 0}$ is
countable.   For each $j\ge 0$, let $\cI_j$ denote
the open interval $(u_j,u_{j+1})$.
To prove Theorem~\ref{thm:main2}, it suffices to establish the
upper bound
\begin{equation}
\label{eq:Up(3)upperbd}
\big|\Upsilon_L^{(3)}(u)\big|\ll_L u^{-2\theta^2}\qquad(u\in\cI_j)
\end{equation}
for every $j\ge 0$, where the implied constant in \eqref{eq:Up(3)upperbd}
may depend on $L$ but is independent of the index $j$.  Indeed, since
$$
G_L^{(3)}(s)=
\int_1^\infty\Upsilon_L^{(3)}(u)u^{-s/2-1/2}du
=\sum_{j\ge 0}\int_{\cI_j}\Upsilon_L^{(3)}(u)u^{-s/2-1/2}du,
$$
the bound \eqref{eq:Up(3)upperbd} implies that the integral
$G_L^{(3)}(s)$ converges absolutely throughout the
half-plane $\{\sigma>1-2\theta^2\}$, uniformly on compact regions, 
and thus $G_L^{(3)}(s)$ is analytic in that region.
Then
$$
F_\alpha^\sharp(r,q;s)
=H_L(s)+O\big(L^{-1}(\sigma+2\theta)^{-1}\big),
$$
where
$$
H_L(s)\defeq e^{\pi ir}\big(F_\infty(s)+G_L^{(1)}(s)
+G_L^{(2)}(s)+G_L^{(3)}(s)\big)
$$
is analytic in the half-plane $\{\sigma>1-2\theta^2\}$.
The sequence $(H_L(s))_{L\ge 1}$ converges uniformly to
$F_\alpha^\sharp(r,q;s)$ on every compact subset
of $\{\sigma>1-2\theta^2\}$, hence $F_\alpha^\sharp(r,q;s)$
is analytic in that half-plane.
Since $\theta\to 1/(2\tau)$ as $\eps\to 0^+$, Theorem~\ref{thm:main2} follows.

It remains to establish \eqref{eq:Up(3)upperbd}.  To this end, put
\begin{align*}
\Omega^+&\defeq\{j\ge 0:\delta_u>u^{-1/2}\log u\text{~for all~}u\in\cI_j\},\\
\Omega^-&\defeq\{j\ge 0:\delta_u<u^{-1/2}\log u\text{~for all~}u\in\cI_j\}.
\end{align*}
By the manner in which the sequence $(u_j)_{j\ge 0}$ is constructed
(especially, see $(ii)$ above), it follows that
every nonnegative integer $j$ lies either in $\Omega^+$ or in $\Omega^-$. 
Moreover, \eqref{eq:case1est(3)} immediately yields \eqref{eq:Up(3)upperbd}
in the case that $j\in\Omega^+$.  Therefore, it remains to show that
\eqref{eq:Up(3)upperbd} holds for integers $j\in\Omega^-$.

Let $j\in\Omega^-$ be fixed.  For all $u\in\cI_j$ we have
$\delta_u<u^{-1/2}\log u$ (by $(ii)$ above) and
$\kappa_u\ne 0$ (since $\delta_u\ne 0$).
Using the estimates
$$
\widetilde \cX_\gamma(\kappa_u)
=\frac{1-\e(-\kappa_u\gamma)}{2\pi i\kappa_u}
=\frac{1-\e(-r)}{2\pi i\kappa_u}\big(1+O(u^{-1/2}\log u)\big)
$$
and
$$
\e(-q\nu_{\kappa_u}\delta_u)=1+O(u^{-1/2}\log u),
$$
we derive that
$$
\Upsilon_L^{(3)}(u)
=\frac{1-\e(-r)}{2\pi i\kappa_u}
e^{-\pi\delta_u^2u}\big(1+O(u^{-1/2}\log u)\big)
\ll \kappa_u^{-1}.
$$
Consequently, to prove \eqref{eq:Up(3)upperbd} it is enough to establish
the lower bound
\begin{equation}
\label{eq:kappalower}
\kappa_u\gg_L u^{2\theta^2}\qquad(u\in\cI_j).
\end{equation}
Let $\Delta_j\defeq\delta_{u_j}$.
Using $(i)$ and $(ii)$ above, and taking into account that
$\delta_u$ is a right-continuous function of $u$ by \eqref{eq:deltau-defn},
we see that $\delta_u=\Delta_j$ and $\delta_u<u^{-1/2}\log u$ for all $u\in\cI_j$.
Put $k_j\defeq\kappa_{u_j}$, so that $\distint{r-k_j\gamma}=\Delta_j$,
and note that $|k_j|\le u_j^\theta L$.   On the other hand, for any integer $k$
with $|k|<u_j^\theta L$ write $|k|=u^\theta L$ with some real number $u<u_j$;
then
$$
\distint{r-k\gamma}\ge\delta_u>\Delta_j.
$$
This argument shows that
\begin{equation}
\label{eq:kjequals}
|k_j|=u_j^\theta L.
\end{equation}

Since $j\in\Omega^-$, the argument given in Case 2 implies that
there is \emph{precisely one} integer $k$ that satisfies both inequalities
\begin{equation}
\label{eq:bogus}
|k|<u_{j+1}^\theta L\mand\distintup{r-k\gamma}\le u_j^{-1/2}\log u_j
\end{equation}
(namely, the integer $k=\kappa_j$).  On the other hand,
using \eqref{eq:strangethings} in combination with
the definition of discrepancy and
Lemma~\ref{lem:discr_with_type}, one sees that the number of integers $k$
satisfying \eqref{eq:bogus} is
$$
2u_j^{-1/2}\log u_j\cdot u_{j+1}^\theta L
+O\((u_{j+1}^\theta L)^{1-1/(\tau+\eps)}\).
$$
For large $j$, this leads to a contradiction unless both bounds
$$
2u_j^{-1/2}\log u_j\cdot u_{j+1}^\theta L\ll 1
$$
and
$$
2u_j^{-1/2}\log u_j\cdot u_{j+1}^\theta L\ll
(u_{j+1}^\theta L)^{1-1/(\tau+\eps)}
$$
satisfied.  We deduce that
\begin{equation}
\label{eq:ujcompare}
u_{j+1}^{2\theta} \ll_L u_j.
\end{equation}
Now let $u\in\cI_j$.  Since
$$
|\kappa_u|\le u^\theta L\mand
\distint{r-\kappa_u\gamma}=\delta_u=\Delta_j,
$$
the integer $k=\kappa_u$ satisfies both inequalities in \eqref{eq:bogus}.
Consequently, $\kappa_u=k_j$, and using \eqref{eq:kjequals}
and \eqref{eq:ujcompare} we have
$$
|\kappa_u|=|k_j|=u_j^\theta L\gg_L u_{j+1}^{2\theta^2}>u^{2\theta^2}.
$$
This is the required bound \eqref{eq:kappalower}, and our proof of
Theorem~\ref{thm:main2} is complete.

\end{document}